\documentclass[11pt,reqno,oneside, a4paper]{amsart}
\usepackage[ansinew]{inputenc}
\usepackage[english]{babel}
\usepackage[T1]{fontenc} 
\usepackage{lmodern}
\usepackage{amsfonts} 		
\usepackage{amsmath}
\usepackage{amssymb}
\usepackage{bbm}
\usepackage{amsthm}
\usepackage{arydshln}
\usepackage{xcolor}
\usepackage{enumitem}
\usepackage{mathabx}
\usepackage{comment}
\usepackage{mathrsfs}
\usepackage{url}
\usepackage{hyperref}
\usepackage{cite}
\usepackage{graphicx}
\usepackage{caption}
\usepackage{algorithm}
\usepackage{algpseudocode}
\usepackage{subcaption}

\usepackage[capitalize]{cleveref}
\crefname{enumi}{}{}
\crefname{equation}{}{}

\usepackage{geometry}

 \specialcomment{supplementary}
{\colorlet{savedcolor}{.} \color{blue} \begingroup \ttfamily \bigskip \smallskip  \noindent \underline{Supplementary details:} \newline \newline \footnotesize }{\endgroup \smallskip \color{savedcolor}}
 \excludecomment{supplementary}

\setenumerate[0]{label=\textnormal{(\roman*)}}

\numberwithin{equation}{section}

\newtheorem{thm}{Theorem}[section]
\newtheorem{definition}[thm]{Definition}

\theoremstyle{remark}

\newtheorem{remark}[thm]{Remark}

\begin{document}

\title[Parameter estimation in  generalized fractional  neuronal models]{Parameter estimation in  generalized fractional  neuronal models}

\author[Ilmonen]{Pauliina Ilmonen}
\address{Aalto University School of Science, Department of Mathematics and Systems Analysis, PO Box 11100, 00076 Aalto, Finland}
\email{pauliina.ilmonen@aalto.fi}

\author[Laurikkala]{Milla Laurikkala}
\address{Aalto University School of Science, Department of Mathematics and Systems Analysis, PO Box 11100, 00076 Aalto, Finland}
\email{milla.laurikkala@aalto.fi}

\author[Pirozzi]{Enrica Pirozzi}
\address{Department of Mathematics and Physics, University of Campania Luigi Vanvitelli, Caserta, 81100, Italy}

\email{enrica.pirozzi@unicampania.it}

\author[Caputo]{Luigia Caputo}
\address{Department of Mathematics and Applications, University of Naples Federico II, Naples, 80126, Italy}
\email{luigia.caputo@unina.it}

\author[Viitasaari]{Lauri Viitasaari}
\address{Aalto University School of Business, Department of Information and Service Management, PO Box 21210, 00076 Aalto, Finland}
\email{lauri.viitasaari@aalto.fi}

\date{\today}

\thanks{
P. Ilmonen and M. Laurikkala gratefully acknowledge support from the Research Council of Finland via Finnish Centre of Excellence in Randomness
and Structures 
(decision number~346308). M. Laurikkala gratefully acknowledges support from Emil Aaltonen Foundation (grant number 250110 K1).}

\begin{abstract}
We investigate a generalized stochastic fractional neuronal model combining
fractional dynamics with correlated stochastic inputs. The proposed framework is
described by a fractional differential equation driven by a latent stochastic process
with stationary increments and mean-reverting structure. This formulation allows the inclusion of both short-range and long-range
dependence structures and naturally produces non-exponential relaxation phenomena.
The main goal is the development of a feasible parameter estimation
procedure based on discrete observations of the neuronal state process. We propose
a two-step methodology. First, the parameters governing the fractional dynamics are
estimated by exploiting the asymptotic behavior of Mittag--Leffler functions near
the origin. Subsequently, the latent stochastic input is reconstructed through
fractional differentiation techniques, allowing the estimation of the parameters
governing the hidden noise dynamics.
We derive quantitative error bounds for the estimators and analyze the reconstruction
error of the latent process under suitable regularity assumptions on the driving
noise. In particular, the interplay between the order of the fractional derivative
and the H\"older regularity of the noise process naturally emerges in the stability
analysis of the reconstruction procedure.
Finally, simulation studies illustrate the applicability of the proposed methodology
and highlight the influence of memory effects and noise regularity on the quality
of statistical inference. The results support the relevance of fractional stochastic
analysis for the modeling and inference of neuronal systems with memory and correlated
inputs.
\end{abstract}

\allowdisplaybreaks


\maketitle

{\footnotesize\tableofcontents}

\section{Introduction}
The mathematical modeling of neuronal activity has a long and rich history, spanning biophysical, probabilistic, and more recently fractional-calculus-based approaches. Classically models are based on Markovian structures, see monograph 
\cite{Tuckwell1988a} and a survey \cite{SacerdoteGiraudo2013}. Recently however, 
the mathematical modeling of neuronal activity has progressively evolved from purely Markovian frameworks toward nonlocal and memory-dependent formulations. The reason behind this is the experimental evidence indicating that neuronal systems exhibit persistent temporal correlations, hereditary effects, adaptation mechanisms, and anomalous relaxation phenomena that cannot be adequately described by classical stochastic differential equations driven by white noise. This has motivated the introduction of stochastic models based on fractional calculus, which have gained increasing attention in the description of complex biological systems where memory effects and nonlocal dynamics play a fundamental role \cite{Magin2006,Mainardi2010,MeerschaertSikorskii2012,MetzlerKlafter2000}.

Fractional derivatives, and in particular the Caputo derivative, provide a convenient mathematical tool to incorporate such memory effects into dynamical systems. In this framework, the evolution of the system depends not only on its current state but also on its entire past history, allowing for a more realistic modeling of biological phenomena. This perspective has been successfully adopted in several contexts, including anomalous diffusion, viscoelasticity, and mathematical neuroscience \cite{Pirozzi2024b,Teka2014}.
In the context of neuronal modeling, fractional stochastic dynamics was investigated in \cite{Pirozzi2018}, where a model driven by colored noise was introduced to describe correlated synaptic inputs and adaptation effects in neuronal firing activity. That model already demonstrated that fractional dynamics naturally produce non-exponential relaxation and persistent correlations consistent with experimentally observed neuronal behavior. Subsequently, the class of fractionally integrated stochastic processes was further developed in \cite{AbundoPirozzi2021,Pirozzi2024}. The processes introduced therein generalize classical Gauss--Markov and Ornstein--Uhlenbeck dynamics by replacing the exponential kernel with Mittag--Leffler type kernels arising from fractional integration, providing analytically tractable models with explicit covariance structures and long-memory properties.
A complementary and closely related framework is provided by time-changed stochastic processes, typically constructed via inverse subordinators, which lead to non-Markovian dynamics with heavy-tailed waiting times and long-range dependence \cite{MeerschaertStraka2013}. These processes offer a flexible framework for modeling subdiffusive behaviors and complex temporal structures, and have been applied to biological systems to capture temporal patterns arising in intracellular dynamics and neuronal activity \cite{LeonenkoPirozzi2025}. In particular, the combination of fractional integration and stochastic time changes leads to models that naturally exhibit both long memory and non-exponential relaxation. 

Despite this growing literature, parameter estimation for neuronal systems with latent stochastic inputs remains a challenging and largely open problem. The main difficulty lies in the simultaneous presence of three features: (i) fractional dynamics, (ii) partially observed systems, and (iii) general noise structures deviating from classical Brownian motion.
In this article, we consider a generalized neuronal model of the form
$$
D^\alpha V_t = A V_t + \eta(t),
$$
where $D^\alpha$ denotes the Caputo fractional derivative and the stochastic input $\eta(t)$  is modeled as a mean-reverting process satisfying
\[
d\eta(t) = \Theta(b-\eta(t))\,dt + \sigma\,dG_t,
\]
with $G$ a stochastic process with stationary increments. This formulation extends classical neuronal models by allowing for both fractional dynamics in the membrane potential and a flexible class of driving noises \cite{Pirozzi2024b}, and can be viewed simultaneously as a generalized fractional neuronal model extending \cite{Pirozzi2018,Pirozzi2024b}, as a generalized fractionally integrated stochastic system in the spirit of \cite{AbundoPirozzi2021,Pirozzi2024}, and as a non-Markovian neuronal model with latent correlated stochastic input. The stochasticity arises from the generalized Ornstein--Uhlenbeck model for the noise $\eta$ that acts as the random noise for the system describing observed $V$. The memory effect can be incorporated to the model via two mechanism; 1) through non-local fractional derivative $D^\alpha$ incorporating memory of $V$ and, 2) through random noise $\eta$ that can possess arbitrary memory structures via driving noise $dG_t$ that essentially is simply assumed to have stationary increments.  This formulation extends classical neuronal models by allowing for both fractional dynamics in the membrane potential and a flexible class of driving noises \cite{Pirozzi2024b}, and can be viewed simultaneously as a generalized fractional neuronal model extending \cite{Pirozzi2018,Pirozzi2024b}, as a generalized fractionally integrated stochastic system in the spirit of \cite{AbundoPirozzi2021,Pirozzi2024}, and as a non-Markovian neuronal model with latent correlated stochastic input.

Our main goal is to develop a feasible parameter estimation procedure for unknown  $(\alpha,A,b,\Theta,\sigma)$ from the observations $V$. Our approach follows a natural two-step strategy: first, we estimate the parameters $(\alpha,A)$ governing the fractional dynamics by exploiting the small-time asymptotic properties of the solution expressed in terms of Mittag--Leffler functions. After parameters $A$ and $\alpha$ are recovered, we reconstruct the latent process $\eta$ and estimate the parameters of its dynamics by using well-established procedures for parameter estimation in Ornstein--Uhlenbeck models, see \cite{Marko1,Vasicek2025} and references therein. In particular, if $\eta$ can be reconstructed accurately, we can follow non-parametric approaches presented in \cite{Marko1,Vasicek2025}  for the estimation of $(b,\Theta,\sigma)$. As our main findings, we show that parameters $(\alpha,A)$ can be recovered from observations $V$ by dense observations in small-time asymptotics, allowing to recover latent noise $\eta$. After that, one can follow procedures described in \cite{Vasicek2025} to recover parameters of the stationary $\eta$ by controlling the error one makes in recovering (a priori unobserved) $\eta$. 

Our contribution is twofold. First, we introduce a generalized fractional neuronal model that encompasses several existing approaches as special cases. Second, we propose a practical estimation procedure allowing to estimate all parameters simultaneously by combining small- and large-time asymptotics. 

The rest of the paper is organized as follows. In Section \ref{sec:theory} we introduce our model and define our estimation procedure together with our main theorems providing quantitative error bounds for the estimators. In Section \ref{sec:simulations} we illustrate our procedure via numerical examples highlighting both its effectiveness and its limitations. We end the paper with short discussions in Section \ref{sec:conclusions}.

\section{Theoretical approach}
\label{sec:theory}
The model is defined via equations 
\begin{eqnarray}\label{Gmod}
 && D^\alpha V_t = AV_t + \eta(t), \qquad V(0)=V_0,
\nonumber\\
&&
\\
&& d\eta(t) = \Theta(b-\eta(t))dt + \sigma dG_t, \qquad \eta(0)=0,\nonumber
 \end{eqnarray}
 where $\alpha \in (0,1), A,b,\Theta \in \mathbb R.$ Here $D^\alpha f$ denotes the Caputo fractional derivative defined as (see, e.g. \cite{MeerschaertStraka2013})
 \begin{equation}\label{eq:Caputo}
	D^{\alpha}f(t)=\frac{1}{\Gamma(1-\alpha)}\int_0^t f^\prime (\tau) (t-\tau)^{-\alpha} d\tau,
\end{equation}
  Without loss of generality, we assume $V_0=1$ for notational simplicity. Moreover, we assume $A<0$, referring to the Leaky Integrate-and-Fire (LIF) neuronal model  based on the stochastic version of the Langevin equation, \cite{Burk}. The solution $V$ is given by  (see \cite{Pirozzi2024})
\begin{align}
    V_t &= E_\alpha(t^\alpha A) + \int_0^t s^{\alpha-1}E_{\alpha,\alpha}(s^\alpha A)\eta(t-s)ds, \quad \forall t>0, \quad V_0=1. \label{eq:sol}
\end{align}
Here $E_{\alpha,\beta}$ is the two-parameter Mittag-Leffler function given by 
\begin{equation}
\label{eq:ML}
E_{\alpha,\beta}(z) = \sum_{k=0}^\infty \frac{z^k}{\Gamma(\alpha k + \beta)}
\end{equation}
and $E_\alpha(z) = E_{\alpha,1}(z)$. 
From \eqref{eq:ML} we obtain asymptotical behaviour at zero given by
$$
E_{\alpha,\beta}(x) \sim 1 + \frac{x}{\Gamma(\alpha + \beta)}.
$$
This motivates the following definition.
\begin{definition}
    We set 
    \begin{equation}
\label{eq:hatalpha}
    \hat\alpha = I_{V_{2t}<1, V_t<1}\frac{\log(1-V_{2t}) - \log(1-V_t)}{\log 2}
\end{equation}
with the convention $\hat\alpha = 0$ if $I_{V_{2t}<1, V_t<1}=0$
and
\begin{equation}
\label{eq:hatA}
    \hat{A} = \Gamma(1+\hat\alpha)t^{-\hat\alpha}(V_t-1).
\end{equation}
\end{definition}
\begin{remark}
\label{rmk:regression}
In practice the estimators can be computed by regressing $\log(1-V_t)$ against $\log t$ and use the asymptotics 
$$
\log(1-V_t) \sim \alpha \log t + \log\left[\frac{|A|}{\Gamma(1+\alpha)}\right]
$$
as we have done in Section \ref{sec:simu}. This provides more stable estimators as asymptotics $t\to 0$ becomes numerically unstable due to discretisation errors.
\end{remark}
Once parameters $\alpha$ and $A$ are recovered, one can retrieve the noise process $\eta$ from 
$$
\eta(t) =D^\alpha V_t - AV_t.
$$
In order to stabilize this further, define 
\begin{equation}
    \tilde{V}_t = V_t - E_{\alpha}(At^\alpha).
\end{equation}
Then from linearity it follows that $\tilde{V}_t$ solves \footnote{Here we use the fact that $E_\alpha(At^\alpha)$ solves $D^\alpha f = Af$, see  \cite{Samko1993}. }
$
D^\alpha \tilde{V}_t = \eta(t)
$
with the initial condition $\tilde{V}_0 = 0$. The solution is given by 
$$
\tilde{V}_t = \frac{1}{\Gamma(\alpha)}\int_0^t s^{\alpha-1}\eta(t-s)ds.
$$
This motivates approximating the noise $\eta$ as 
\begin{equation}
\label{eq:eta-proxy}
    \hat\eta(t) = D^{\hat\alpha}\left[V_t - E_{\hat\alpha}(\hat{A}t^{\hat\alpha})\right].
\end{equation}
If $\hat\eta$ is not too far from the true noise $\eta$, then one can recover parameters $(b,\Theta,\sigma)$ exactly as in \cite{Vasicek2025}. More precisely, controlling the error $|\hat\eta(t) - \eta(t)|$ allows to base estimators of $(b,\Theta,\sigma)$ to $\hat\eta$ instead of $\eta$, and the growth of the error $|\hat\eta(t) - \eta(t)|$ (as time window $T$ increases required for the estimation) has to be compensated by better estimation of $\alpha$ and $A$. Our first main result, Theorem \ref{thm:a_alpha_consistency} below, provides precise error analysis on the estimation of $\alpha$ and $A$ on small-time asymptotics. Our second main theorem, Theorem \ref{thm:noise-error} below, provides quantitative error bounds in the estimation $\hat\eta$ of the noise term in terms of error in the estimation of $\alpha$ and $A$. Combined with exact rate of convergences for the estimators for $(b,\Theta,\sigma)$ provided in \cite{Vasicek2025} together with the error rates in Theorem \ref{thm:a_alpha_consistency}, one can now easily deduce how accurately $\alpha$ and $A$ has to be estimated on small-time asymptotics in order to have small enough error for $|\hat\eta-\eta|$ also for large values which is required to estimate $(b,\Theta,\sigma)$ accurately.
\begin{thm}\label{thm:a_alpha_consistency}
    Suppose that $\eta$ is almost surely H\"older continuous of order $H\in(\alpha,1)$. Then there exists a random constant $C=C_\omega$ depending only on $A,\alpha$, and $H$ such that, for all $t\in [0,1]$: 
    \begin{enumerate}
    \item We have \begin{equation}
    \label{eq:alpha-consistent}
    |\hat\alpha - \alpha|\leq C_\omega t^{\alpha}.
    \end{equation}
    In particular, $\hat\alpha \to \alpha$ almost surely as $t\to 0$. 
    \item We have
    \begin{equation}
    \label{eq:A-consistent}
    |\hat{A} -A| \leq C_\omega t^\alpha |\log t|.
    \end{equation}
    In particular, $\hat{A} \to A$ almost surely as  $t\to 0$.
    \end{enumerate}
\end{thm}
\begin{proof}
In the following $C_\omega$ denotes a random constant that might depends on $A,\alpha$, $H$, and $\omega$, but is independent of $t \in [0,1]$. We write 
\begin{equation}
\label{eq:V-asymptotics}
1- V_t = \frac{-At^\alpha}{\Gamma(1+\alpha)} - E(t)
\end{equation}
with
$$
E(t) = E_\alpha(t^\alpha A) - \frac{At^\alpha}{\Gamma(1+\alpha)} + \int_0^t s^{\alpha-1}E_{\alpha,\alpha}(s^\alpha A)\eta(t-s)ds.
$$
We first prove an upper bound for $E(t)$. For this, we get from \eqref{eq:ML} that
\begin{equation}
\label{eq:deterministic-error}
\left|E_\alpha(t^\alpha A)-1 - \frac{At^\alpha}{\Gamma(1+\alpha)}\right| \leq Ct^{2\alpha}.
\end{equation}
For the random part, H\"older continuity of $\eta$ implies 
\begin{align*}
&\left\vert\int_0^t s^{\alpha-1}E_{\alpha,\alpha}(s^\alpha A)\eta(t-s)ds \right\vert \leq C_\omega \int_0^t s^{\alpha-1}(t-s)^Hds \\
& \leq C_\omega t^{H+\alpha} \leq C_\omega t^{2\alpha}.
\end{align*}
Combining the above estimates gives us 
\begin{equation}
\label{eq:E_bound}
|E(t)|\leq (C + C_\omega)t^{2\alpha}.
\end{equation}
This in particular implies that there exists $t_0=t_0(\omega)$ such that for all $t\leq t_0$ we have 
\begin{equation}
\label{eq:E_bound-chosen}
|E(t)| \leq \frac{-At^\alpha}{2\Gamma(1+\alpha)},
\end{equation}
which in turn implies 
$$
1-V_t = \left[\frac{-A}{\Gamma(1+\alpha)} - \frac{E(t)}{t^\alpha}\right]t^\alpha > 0.
$$
Consequently, for $t\leq \frac{t_0}{2}$ we have, almost surely, that 
$$
\hat\alpha = \frac{\log(1-V_{2t}) - \log(1-V_t)}{\log 2}.
$$
Note that here 
\begin{align*}
\log (1-V_t) &= \log\left[\frac{-At^\alpha}{\Gamma(1+\alpha)} - E(t)\right] \\
& = \log\left[\frac{-A}{\Gamma(1+\alpha)}\right] + \alpha \log t + \log\left[1+\frac{\Gamma(1+\alpha)E(t)}{At^\alpha}\right].
\end{align*}
Similarly we obtain 
\begin{align*}
\log (1-V_{2t}) &= \log\left[\frac{-A(2t)^\alpha}{\Gamma(1+\alpha)} - E(2t)\right] \\
& = \log\left[\frac{-A}{\Gamma(1+\alpha)}\right] + \alpha \log 2 + \alpha \log t + \log\left[1+\frac{\Gamma(1+\alpha)E(2t)}{A(2t)^\alpha}\right]
\end{align*}
leading to
\begin{align*}
    |\hat\alpha - \alpha| &= \left|\frac{\log(1-V_{2t}) - \log(1-V_t)}{\log 2} - \alpha\right| \\
    &=\left|\log\left[1+\frac{\Gamma(1+\alpha)E(2t)}{A(2t)^\alpha}\right] - \log\left[1+\frac{\Gamma(1+\alpha)E(t)}{At^\alpha}\right]\right|\\
    &\leq \left|\log\left[1+\frac{\Gamma(1+\alpha)E(2t)}{A(2t)^\alpha}\right]\right| + \left|\log\left[1+\frac{\Gamma(1+\alpha)E(t)}{At^\alpha}\right]\right|.
\end{align*}
Now \eqref{eq:E_bound-chosen} and $t\leq \frac{t_0}{2}$ allows us to use $|\log(1+x)|\leq \frac32|x|$ for $|x|\leq \frac12$. This together with \eqref{eq:E_bound} gives an estimate
$$
|\hat\alpha - \alpha| \leq C_\omega t^{\alpha}
$$
proving \eqref{eq:alpha-consistent}. For \eqref{eq:A-consistent}, we use \eqref{eq:V-asymptotics} and \eqref{eq:hatA} to obtain
\begin{align*}
    \hat{A} - A &= \left[\frac{\Gamma(1+\hat\alpha)t^\alpha}{\Gamma(1+\alpha)t^{\hat\alpha}} - 1\right]A + \frac{\Gamma(1+\hat\alpha)}{t^{\hat\alpha}}E(t)\\
    &= \left[\frac{\Gamma(1+\hat\alpha)}{\Gamma(1+\alpha)} - 1\right]t^{\alpha-\hat\alpha}A + \left[t^{\alpha-\hat\alpha} - 1\right]A + \frac{\Gamma(1+\hat\alpha)}{t^{\hat\alpha}}E(t).
\end{align*}
We first analyze $t^{\alpha-\hat\alpha}$ and provide rate at which it converges to a constant $1$. Using \eqref{eq:alpha-consistent} allows to estimate 
\begin{align*}
   \left| t^{\alpha-\hat\alpha} - 1\right| &= \left|e^{(\alpha-\hat\alpha) \log t}-1\right| \leq e^{|\alpha-\hat\alpha||\log t|} - 1 \leq e^{C_\omega t^\alpha |\log t|} - 1 \\
   &\leq C_\omega t^\alpha |\log t|e^{C_\omega t^\alpha |\log t|}\leq C_\omega t^\alpha |\log t|.
\end{align*}
Using also smoothness of $\Gamma(1+x)$ around $x=\alpha$ and \eqref{eq:E_bound} yields immediately for the last two terms that 
\begin{align*}
    &\left[t^{\alpha-\hat\alpha} - 1\right]A + \frac{\Gamma(1+\hat\alpha)}{t^{\hat\alpha}}E(t)  \leq C_\omega t^\alpha |\log t| + C_\omega t^{\alpha}t^{\alpha-\hat\alpha}\leq C_\omega t^\alpha |\log t|.
\end{align*}
Similarly, using smoothness of the Gamma function $\Gamma(x)$ we deduce 
\begin{align*}
    &\left[\frac{\Gamma(1+\hat\alpha)}{\Gamma(1+\alpha)} - 1\right]t^{\alpha-\hat\alpha}A\leq C_\omega |\alpha - \hat\alpha|\leq C_\omega t^\alpha.
\end{align*}
Combining all the estimates leads to \eqref{eq:A-consistent} and completes the whole proof.
\end{proof}
By Theorem \ref{thm:a_alpha_consistency} we can recover $\alpha$ and $A$ with arbitrary precision by zooming in to short-time asymptotics. In order to estimate $(b,\Theta,\sigma)$, note that $\sigma$ can also be estimated with arbitrary precision on short time asymptotics if $\eta$ can be recovered, see \cite{Vasicek2025}. The estimation of $(b,\Theta)$ on the other hand is based on long periods $[0,T]$ with $T\to \infty$ allowing accurate estimation of the mean and covariances of an (asymptotically) stationary series $\eta$. For this purpose one also needs the control of the error $|\hat\eta-\eta|$ that is the topic of the next theorem. 

In the sequel, we use notation $f\lesssim g$ if $|f|\leq C|g|$ for some unimportant constant $C$ independent of $T$, but possibly depending on true parameters $(\alpha,A,b,\Theta,\sigma)$. We also note that, as time interval $[0,1]$ is dedicated to the estimation of parameters $(\alpha,A)$, the noise recovery is considered after time period $[0,1]$  in the following.

\begin{thm}
\label{thm:noise-error}
Suppose that $\eta$ is almost surely H\"older continuous of order $H\in(\alpha,1)$ and that $A<0$. Let $\hat\eta$ be given by \eqref{eq:eta-proxy}. Then for any fixed $T>1$ we have 
    $$
    \sup_{1\leq t\leq T}|\hat{\eta}(t) - \eta(t)| \lesssim \max(1,C^\eta_T)T^{H-\alpha}\log T |\alpha - \hat\alpha| + |A - \hat{A}|,
    $$
    where $C^\eta_T$ is the H\"older constant of $\eta$ on $[0,T]$.
\end{thm}
\begin{proof}
Starting from the elementary computations
$$
D^\mu E_{\alpha}(A t^\alpha) = At^{\alpha-\mu}E_{\alpha,\alpha+1-\mu}(A t^\alpha)
$$
we obtain
\begin{align}\label{etaapprox}
    \hat{\eta}(t) - \eta(t) &= (D^{\hat\alpha} - D^\alpha)V_t + D^{\alpha}E_\alpha(At^\alpha) - D^{\hat\alpha}E_{\hat\alpha}(\hat{A}t^{\hat\alpha}) \nonumber \\
    &= (D^{\hat\alpha} - D^\alpha)V_t + AE_\alpha(At^\alpha) - \hat{A}E_{\hat\alpha}(\hat{A}t^{\hat\alpha})\nonumber \\
    &=[At^{\alpha-\hat\alpha}E_{\alpha,\alpha+1-\hat\alpha}(A t^\alpha) - \hat{A}E_{\hat\alpha}(\hat{A}t^{\hat\alpha})]+ (D^{\hat\alpha} - D^\alpha)f(t)
\end{align}
with 
\begin{equation*}
f(t) = \int_0^t (t-s)^{\alpha-1}E_{\alpha,\alpha}(A(t-s)^\alpha)\eta(s)ds.
\end{equation*}
Now standard regularity properties of
fractional integral operators imply that $f$ inherits the H\"older regularity of $\eta$. Indeed, in other words,  if $\eta$ possesses $H$-H\"older continuous trajectories on compact intervals, then standard regularity properties of fractional convolution operators imply that $f$ is also $H$-H\"older continuous on compact intervals, with H\"older constant controlled by that of $\eta$.

We begin with the second term in \eqref{etaapprox} that is trickier. Since $f$ is H\"older continuous, we have 
\begin{equation*}
D^\alpha f(t) = \frac{1}{\Gamma(1-\alpha)}\frac{f(t)-f(0)}{t^\alpha} + \frac{\alpha}{\Gamma(1-\alpha)}\int_0^t \frac{f(t)-f(s)}{(t-s)^{1+\alpha}}ds.
\end{equation*}
Set 
$$
\hat{D}^\alpha f(t) = \frac{f(t)-f(0)}{t^\alpha} + \int_0^t \frac{f(t)-f(s)}{(t-s)^{1+\alpha}}ds.
$$
Using smoothness of the Gamma function $\Gamma(x)$ around true $\alpha \in (0,1)$ we obtain 
\begin{align*}
& |(D^{\hat\alpha} - D^\alpha)f(t)| \\
&\leq  \left|\frac{1}{\Gamma(1-\hat\alpha)} - \frac{1}{\Gamma(1-\alpha)}\right|\frac{|f(t)-f(0)|}{t^\alpha} + \frac{1}{\Gamma(1-\alpha)}|\hat{D}^\alpha f(t) - \hat{D}^{\hat\alpha} f(t)|\\
&+  \left|\frac{\hat\alpha}{\Gamma(1-\hat\alpha)} - \frac{\alpha}{\Gamma(1-\alpha)}\right||\hat{D}^\alpha f(t) - \hat{D}^{\hat\alpha} f(t)|\\
&\lesssim |\hat\alpha - \alpha||\hat{D}^\alpha f(t)| + |\hat{D}^\alpha f(t) - \hat{D}^{\hat\alpha} f(t)|.
\end{align*}
Here 
$$
|\hat{D}^\alpha f(t)| \lesssim C^\eta_T T^{H-\alpha}
$$
handling the first term. For the second term, applying H\"older continuity together with 
$$
\left\vert\frac{1}{t^{\alpha}} - \frac{1}{t^{\hat\alpha}}\right\vert \lesssim \frac{\log t}{t^\alpha} |\alpha - \hat\alpha|
$$
yields
$$
|(D^{\hat\alpha} - D^\alpha)f(t)| \lesssim C_T^\eta \log T |\alpha-\hat\alpha| T^{H-\alpha}.
$$
It remains to study the first term in square brackets in \eqref{etaapprox}, i.e.
\begin{align*}
& At^{\alpha-\hat\alpha}E_{\alpha,\alpha+1-\hat\alpha}(A t^\alpha) - \hat{A}E_{\hat\alpha}(\hat{A}t^{\hat\alpha}).
\end{align*}
We decompose it as
\begin{align*}
	&
	A t^{\alpha-\hat\alpha}
	E_{\alpha,\alpha+1-\hat\alpha}(At^\alpha)
	-
	\hat A
	E_{\hat\alpha}( \hat A t^{\hat\alpha} )
	\\
	&=
	A\Big(
	t^{\alpha-\hat\alpha}-1
	\Big)
	E_{\alpha,\alpha+1-\hat\alpha}(At^\alpha)
	\\
	&\quad
	+
	A\Big(
	E_{\alpha,\alpha+1-\hat\alpha}(At^\alpha)
	-
	E_\alpha(At^\alpha)
	\Big)
	\\
	&\quad
	+
	(A-\hat A)E_\alpha(At^\alpha)
	\\
	&\quad
	+
	\hat A
	\Big(
	E_\alpha(At^\alpha)
	-
	E_{\hat\alpha}(\hat A t^{\hat\alpha})
	\Big) \\
    &=: J_1+J_2+J_3+J_4.
\end{align*}
Since $A<0$, $\sup_{t\geq1}
|E_\alpha(At^\alpha)|
<\infty$ and hence proceeding as in the proof of Theorem \ref{thm:a_alpha_consistency} we get
$$
|J_1|\lesssim |\alpha - \hat\alpha|\log T.
$$
For $J_2$ and $J_3$, continuity properties and boundedness of the Mittag--Leffler functions for $A<0$ yields 
$
|J_2| \lesssim |\alpha-\hat\alpha|
$
and
$
|J_3| \lesssim |A- \hat{A}|.
$
Working similarly we get also
$$
|J_4| \lesssim |A - \hat{A}|+ |\alpha - \hat\alpha|
$$
and collecting all the estimates completes the proof.

\end{proof}
\begin{remark}
    The estimation of $\sigma$ now follows directly by following the approach based in quadratic variations on short intervals presented in \cite{Vasicek2025}. The estimation of $(b,\Theta)$ however is based on the estimation of the mean and covariances of a (almost) stationary series $\eta$, as discussed in \cite{Vasicek2025}. For this however, we note that the error in $|\hat\eta -\eta|$ propagates as $T$ increases, while observed series length $T$ has to increase in order to estimate mean and covariance of $\eta$ accurately. Together with Theorem \ref{thm:noise-error} and asymptotical behaviour of estimators $(\hat{b},\hat{\Theta})$ this allows to explicitly determine how accurately $(\alpha,A)$ has to be estimated (on short-time asymptotics) in order to obtain sufficient accuracy for the estimation of $(b,\Theta)$ based on large-time asymptotics.
\end{remark}

\section{Numerical illustrations}
\label{sec:simulations}
In this section we illustrate the performance of our approach in practice via simulated examples. In Section \ref{sec:simu} we describe our simulation setting, while in Section \ref{sec:estimation} we present our results.
\subsection{Simulations}
\label{sec:simu}
We adopt the strategy for simulating sample paths of $V$ from \cite[Section 3]{Pirozzi2018} (see also \cite[Section 6]{LeonenkoPirozzi2025}). Indeed, motivated by the definition \eqref{eq:Caputo}, for given discretization mesh with interval length $\Delta t$ one sets for $t=N \Delta t$ 
\begin{equation}\label{approxCaputo}
	D^\alpha V(t) \approx \frac{(\Delta t) ^{-\alpha}}{\Gamma(2-\alpha)} \sum_{k=0}^{N-1}[V((k+1)\Delta t)-V(k\Delta t)][(N-k)^{1-\alpha}-(N-k-1)^{1-\alpha}].
\end{equation}
In our simulations we consider the noise $G$ in the noise dynamics of \eqref{Gmod} to be the fractional Brownian motion (fBm) $B^H$ with Hurst index $H$.That is, $B^H$ is a centered Gaussian process with $B^H_0=0$ and covariance 
$$
R(s,t) = \frac{1}{2}\left[t^{2H} + s^{2H} - |t-s|^{2H}\right].
$$
Within this noise, the process $V$ can be simulated through Algorithm 1. Figure \ref{fig:paths_different} shows examples of simulated paths of $V$ for different values of the parameters $A$ and $\Theta$.

\begin{algorithm}
\label{algorithm:1}
\caption{Simulation of the fractional neuronal model}
\begin{algorithmic}[1]
\item \textbf{Initialize} the parameters $\alpha, A, b, \Theta, \sigma, H$, 
    the initial value $V_0$, the step size $\Delta t$, the maximum number 
    of time steps $N$, and the number of paths $M$.

\item Set the seed of the random number generator.

\item  \textbf{Simulate the OU process} $\eta(t)$ via an Euler--Maruyama 
    discretization scheme. At time instants 
    $0 = t_0 < t_{\Delta t} < t_{2\Delta t} < \cdots < t_{N\Delta t}$, 
    simulate a fBm $B^H(t)$ and construct the path 
    of $\eta(t)$ by setting $\eta(0) = 0$, $B^H(0) = 0$, and
    \begin{equation*}
        \eta(n\Delta t) = \eta((n-1)\Delta t) 
        +(b- \Theta)\,\eta((n-1)\Delta t)\,\Delta t 
        + \sigma\!\left(B^H(n\Delta t) - B^H((n-1)\Delta t)\right)
    \end{equation*}
    for $n = 2, \ldots, N$.

\item \textbf{Construct the path of} $V(t)$ at the same time instants 
    by setting $V(0) = V_0$,
    \begin{equation*}
        V(\Delta t) = V(0) 
        + \frac{(\Delta t)^{\alpha}\,\Gamma(2-\alpha)}{2}
        \left(A\,V(0) + b + \eta(\Delta t)\right),
    \end{equation*}
    and applying the iterative scheme
    \begin{align*}
        V(n\Delta t) 
        &= V((n-1)\Delta t) \\
        &\quad - \sum_{k=0}^{n-2}\bigl[V((k+1)\Delta t) - V(k\Delta t)\bigr]
          \bigl[(n-k)^{1-\alpha} - (n-k-1)^{1-\alpha}\bigr] \\
        &\quad + (\Delta t)^{\alpha}\,\Gamma(2-\alpha)
          \left(A\,V((n-1)\Delta t) + b + \eta(n\Delta t)\right)
    \end{align*}
    for $n = 2, \ldots, N$.

\item Repeat Steps 2--4 a total of $M$ times, setting a new seed for the 
    random number generator at each repetition.
\end{algorithmic}
\end{algorithm}

\begin{figure}[h!]
    \begin{subfigure}{0.495\textwidth}
        \centering
        \includegraphics[width=\textwidth]{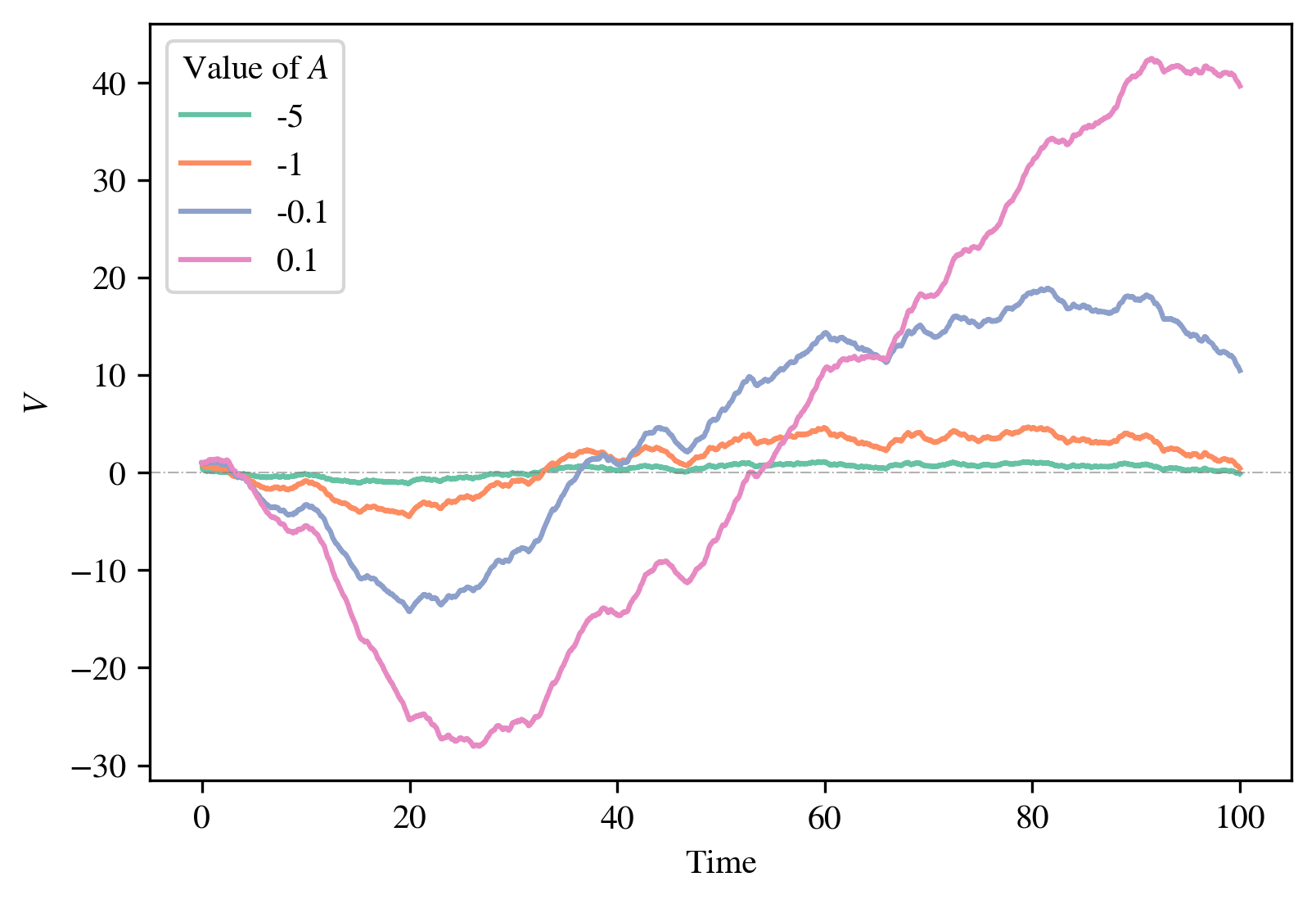}
        \caption{$\Theta = 0.1$.}
        \label{fig:path_theta_0.1}
        \end{subfigure}
        \hfill
    \begin{subfigure}{0.495\textwidth}
        \centering
        \includegraphics[width=\textwidth]{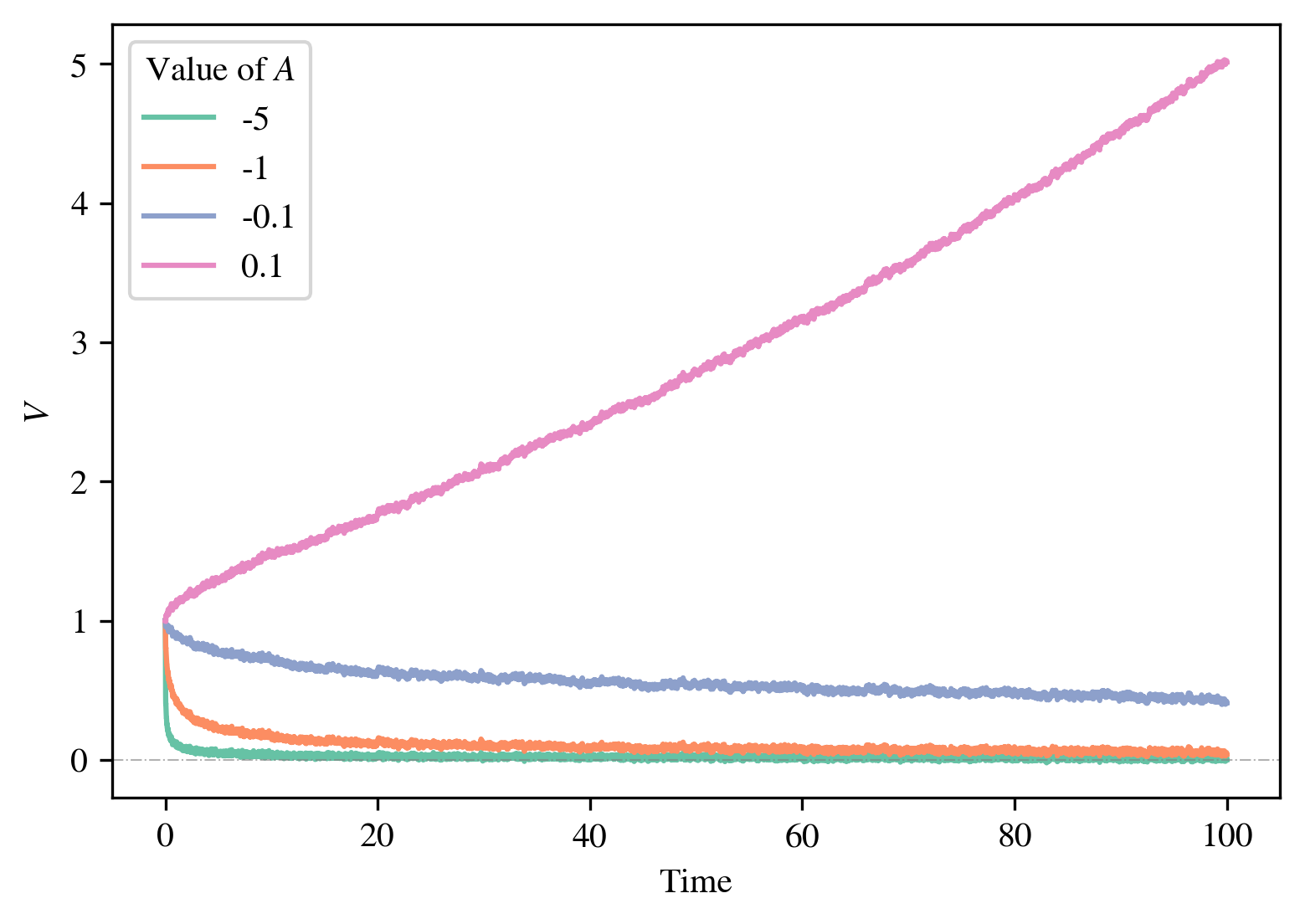}
        \caption{$\Theta = 100$.}
        \label{fig:path_theta_100}
    \end{subfigure}
    \caption{Simulated paths of $V$ for different values of $A$ and $\Theta$.}
    \label{fig:paths_different}
\end{figure}

\subsection{Estimation}
\label{sec:estimation}
We present simulated examples to illustrate the performance and applicability of the proposed estimators.
The simulations were performed using computer resources within the Aalto University School of Science \lq\lq Science-IT\rq\rq\ project. For simplicity, we set $b=0$ and focused on the estimation of $(\alpha,A,\Theta,\sigma)$.

For the estimation of $\alpha$ and $A$, many observations very close to 0 are required, whereas the estimation of $\Theta$ and $\sigma$ requires observations within a long time period.
Due to computational constraints, it is not feasible to work with a very long time series with extremely dense observations, but in practice, one would sample more densely during a short interval in the beginning and then collect a longer time series with less dense observations.
To mimic this practical approach, we simulated two separate paths of $V$ for each sample: one on the time interval $[0, 0.001]$ with 10000 observations, and another on the time interval $[0, 400]$ with 10000 observations.
The first part exhibits asymptotic behavior close to 0 is used to estimate $\alpha$ and $A$, and the second part shows long-term behavior and is used to estimate $\Theta$ and $\sigma$.

In total, 400 samples were simulated for each example.
The simulations were done using the approach presented in Section \ref{sec:simu} and the simulation parameters are given in Table \ref{table:params}.
In the table, $V_0$ denotes the initial value of the process $V$.
The process $G$ in the noise process $\eta$ was a fractional Brownian motion with Hurst index $H$.
In the examples, the values of $\alpha$ and $H$ were varied while other parameters were kept constant.

\begin{table}[h]
\caption{Parameter values for simulations.}
\label{table:params}
\begin{tabular}{cc|cc|cc}
Parameter & Value & Parameter  & Value & Parameter  & Value \\ \hline
$A$       & -0.1   & $\Theta$   & 5     & T          & 400     \\
$b$       & 0     & $\sigma$   & 1   & n          & 10000  \\
$\alpha$  & \{0.5, 0.7\}   & $H$        & \{0.5, 0.6, 0.7, 0.8\}   & $V_0$      & 1
\end{tabular}
\end{table}

The parameters $\alpha$ and $A$ were estimated from the first part of each sample using the regression method from Remark \ref{rmk:regression} and \eqref{eq:hatA}.
After estimating $\alpha$ and $A$, the noise process $\eta$ was recovered from the longer time series, and the parameters $\Theta$ and $\sigma$ were estimated using the methods of \cite{Vasicek2025}.

The case $\alpha = 0.5$ was simulated with two different noise processes: standard Brownian motion ($H=0.5$) and fBm with Hurst index 0.6.
Figure \ref{fig:alpha_A_hist_0.5} displays the differences of the estimates $\hat \alpha$ and $\hat A$ from their true values as histograms in the case 
The first column of the figure corresponds to the case when $H=0.5$ and the second column to $H=0.6$.
Based on the histograms, it seems that the parameters $\alpha$ and $A$ could be well estimated from the simulated data with both noise processes.
All histograms are centered at 0, but for larger $H$, the histograms appear wider.
This is due to the differences in the correlation structures of $\eta$ with different Hurst indices.

\begin{figure}[ht]
     \centering
     \begin{subfigure}{0.9\textwidth}
         \centering
         \includegraphics[width=\textwidth]{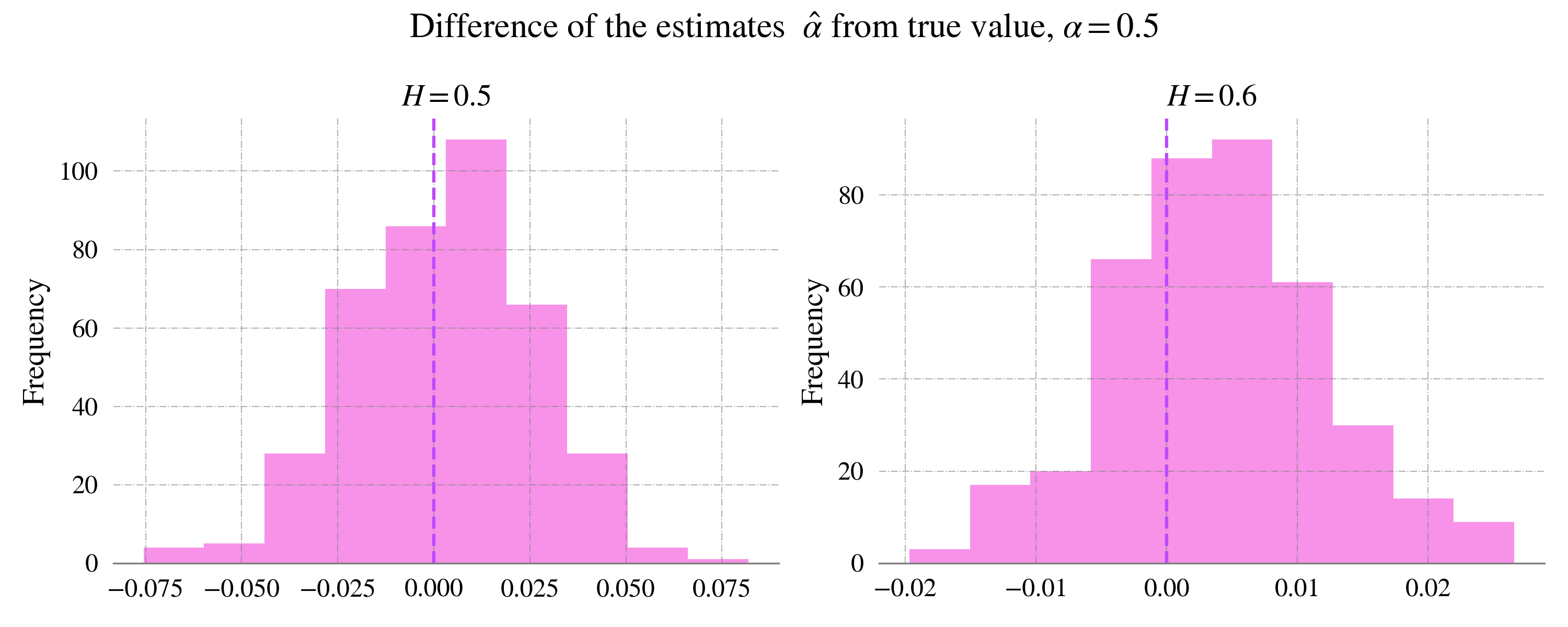}
         \caption{Parameter $\alpha$.}
     \end{subfigure}
     \hfill
     \begin{subfigure}{0.9\textwidth}
         \centering
         \includegraphics[width=\textwidth]{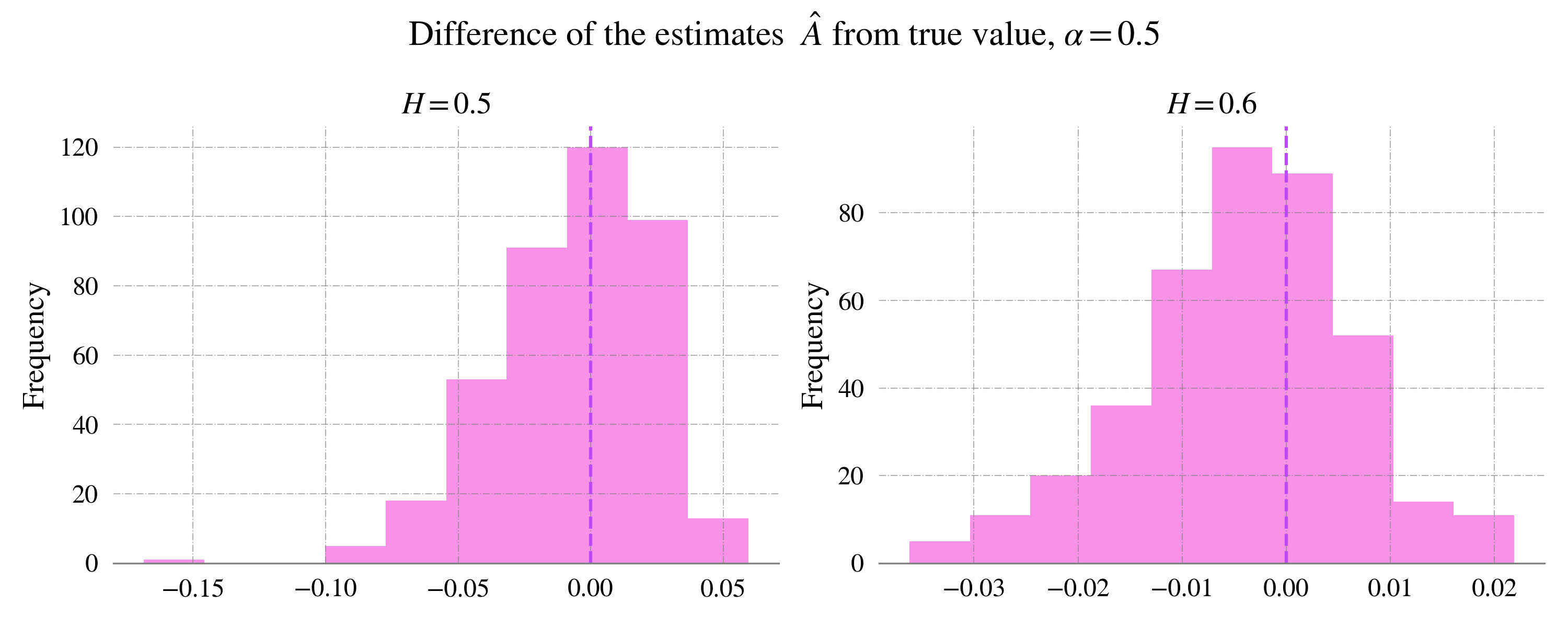}
         \caption{Parameter $A$.}
     \end{subfigure}
        \caption{Differences between parameter estimates and true values of the parameters $\alpha$ and $A$ for different values of $H$ when $\alpha=0.5$.}
        \label{fig:alpha_A_hist_0.5}
\end{figure}

Figure \ref{fig:theta_hist_0.5} shows the differences of the estimates $\hat \Theta$ computed from the recovered noise $\hat\eta$ and the true value $\Theta$ as histograms for the case $\alpha=0.5$.
For comparison, the parameter $\Theta$ was also estimated from the true noise $\eta$, yielding estimates $\hat \Theta_{\text{true noise}}$ for each sample.
The figure also shows the difference of the estimates $\hat \Theta_{\text{true noise}}$ from the true parameter $\Theta$ as histograms.
Based on the histograms, $\Theta$ could be well estimated from the data.
The histograms corresponding to the estimates from recovered noise (left column of the figure) are slightly wider than the histograms corresponding to true noise, especially when $H=0.6$.
However, it seems that the noise recovery works sufficiently well and does not make too large an error.
The parameter $\sigma$ could be estimated equally well from the data, as can be seen in Figure \ref{fig:sigma_hist_0.5}, which shows the differences between the estimates of $\sigma$ and the true value of the parameter in the case $\alpha=0.5$.
The estimates for $\sigma$ were computed from $\hat\eta$.

\begin{figure}[h!]
\includegraphics[width=0.9\textwidth]{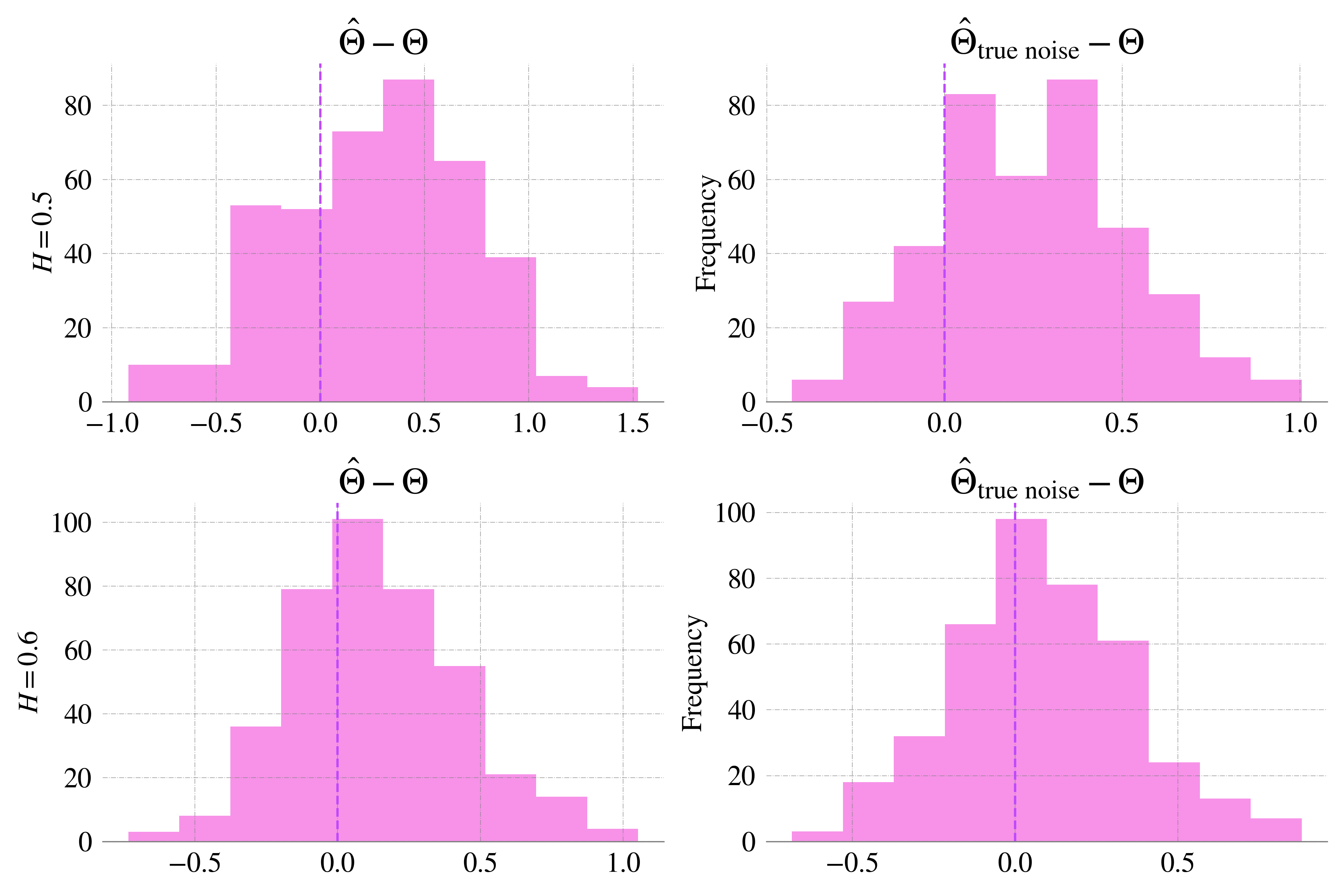}
\caption{Differences between the estimates $\hat \Theta$ from the true parameter and the estimates $\hat \Theta_{\text{true noise}}$ calculated from the real data in the case $\alpha=0.5$.}
\label{fig:theta_hist_0.5}
\end{figure}

For the case $\alpha = 0.7$, simulations were run with two different fractional Brownian motions as the noise process, corresponding to Hurst indices $H=0.7$ and $H=0.8$.
Figure \ref{fig:alpha_A_hist_0.7} displays the differences between the estimates of $\alpha$ and $A$ from their true values for different values of $H$ when $\alpha=0.7$.
The histograms seem to be slightly shifted from 0, especially in the case $H=0.8$.
However, the scale of the horizontal axis of the histograms is, for both $H=0.7$ and $H=0.8$, considerably smaller than in the previous examples with $\alpha=0.5$.
Therefore, it seems that as $\alpha$ increases, the estimates of $\alpha$ and $A$ get more accurate, even though there appears to be a slight bias in the estimates.

\begin{figure}[ht]
     \centering
     \begin{subfigure}{0.9\textwidth}
         \centering
         \includegraphics[width=\textwidth]{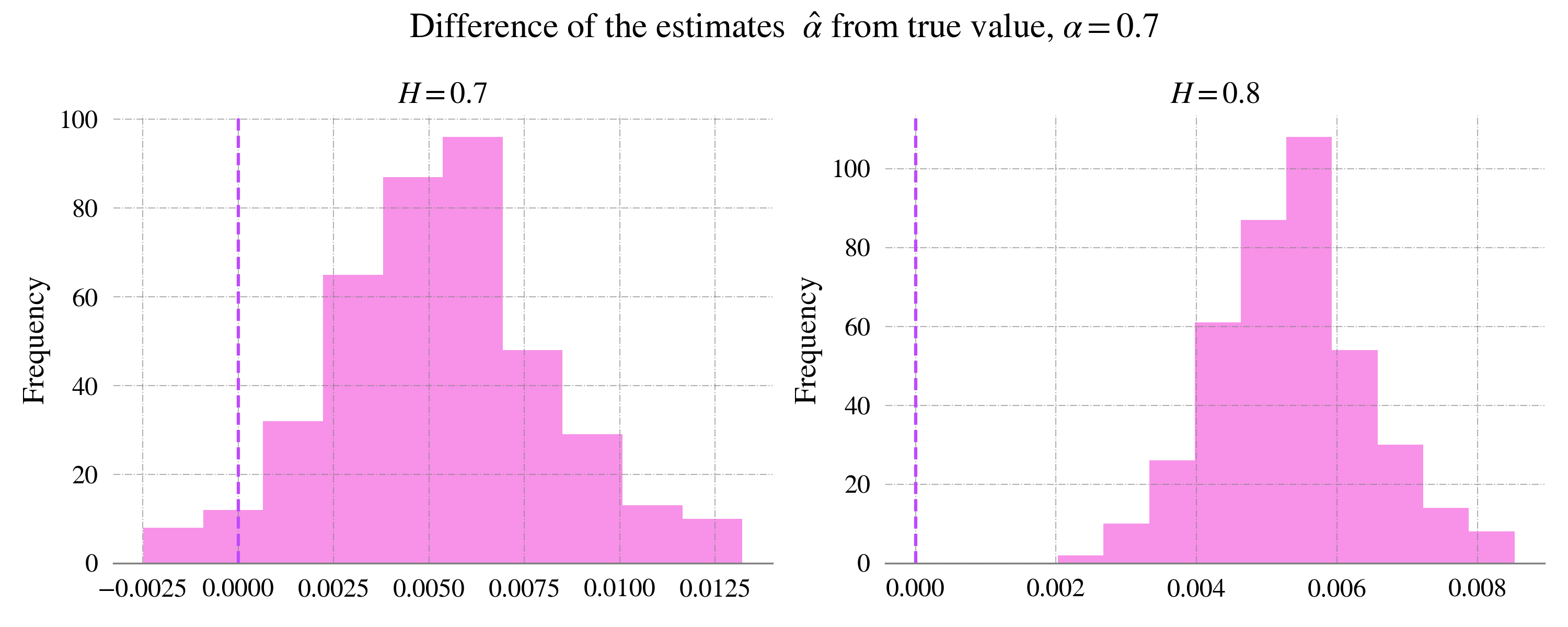}
         \caption{Parameter $\alpha$.}
     \end{subfigure}
     \hfill
     \begin{subfigure}{0.9\textwidth}
         \centering
         \includegraphics[width=\textwidth]{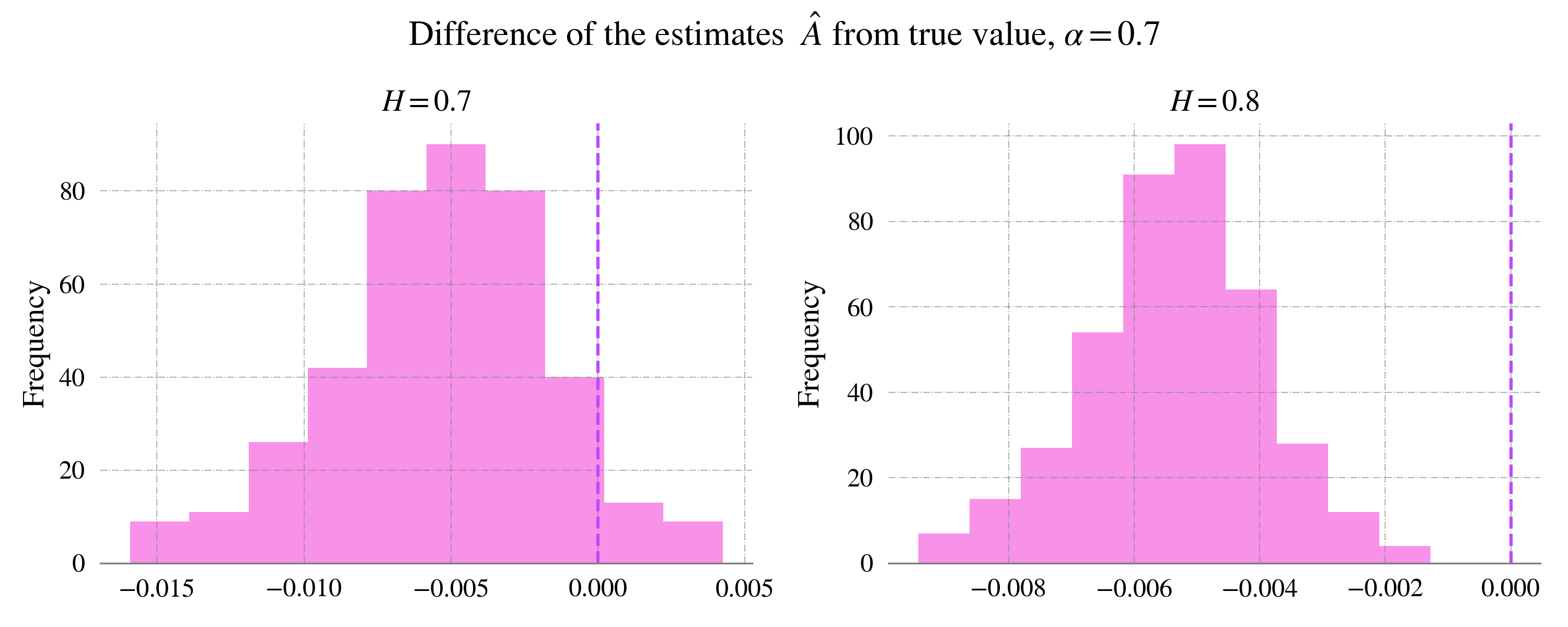}
         \caption{Parameter $A$.}
     \end{subfigure}
        \caption{Differences between parameter estimates and true values of the parameters $\alpha$ and $A$ for different values of $H$ when $\alpha=0.7$.}
        \label{fig:alpha_A_hist_0.7}
\end{figure}

Figure \ref{fig:theta_hist_0.7} displays the differences of the estimates $\hat \Theta$, computed from recovered noise, and the estimates $\hat \Theta_{\text{true noise}}$, computed from true noise, from the true value $\Theta$ as histograms for the case $\alpha=0.7$.
Based on the histograms, $\Theta$ could be well estimated from the data in this case, too.
Now, the histograms corresponding to estimates from recovered noise look very similar to the histograms corresponding to estimates from true noise for both values of $H$. 
The histograms corresponding to recovered noise are still slightly wider than the histograms corresponding to true noise, but this phenomenon is less visible than in the case $\alpha = 0.5$.
The reason for this is that when $\alpha=0.7$, the parameters $\alpha$ and $A$ could be estimated more accurately, resulting in the recovered noise being closer to the true noise.

Figure \ref{fig:sigma_hist_0.7} displays the histograms of the differences between the estimates $\hat \sigma$ and the true value $\sigma$ in the case $\alpha=0.7$.
There is a noticeable bias in the estimates, especially when $H=0.8$. This is aligned with theoretical error provided in Theorem \ref{thm:noise-error} suggesting larger error if $H$ increases compared to $\alpha$. 
However, the magnitude of the bias is small compared to the true value of the parameter (around 1 \% for $H=0.7$ and 6 \% for $H=0.8$), meaning that the parameter could be estimated well.
The same phenomenon was observed in \cite{Vasicek2025}.

\begin{figure}[h!]
\includegraphics[width=0.9\textwidth]{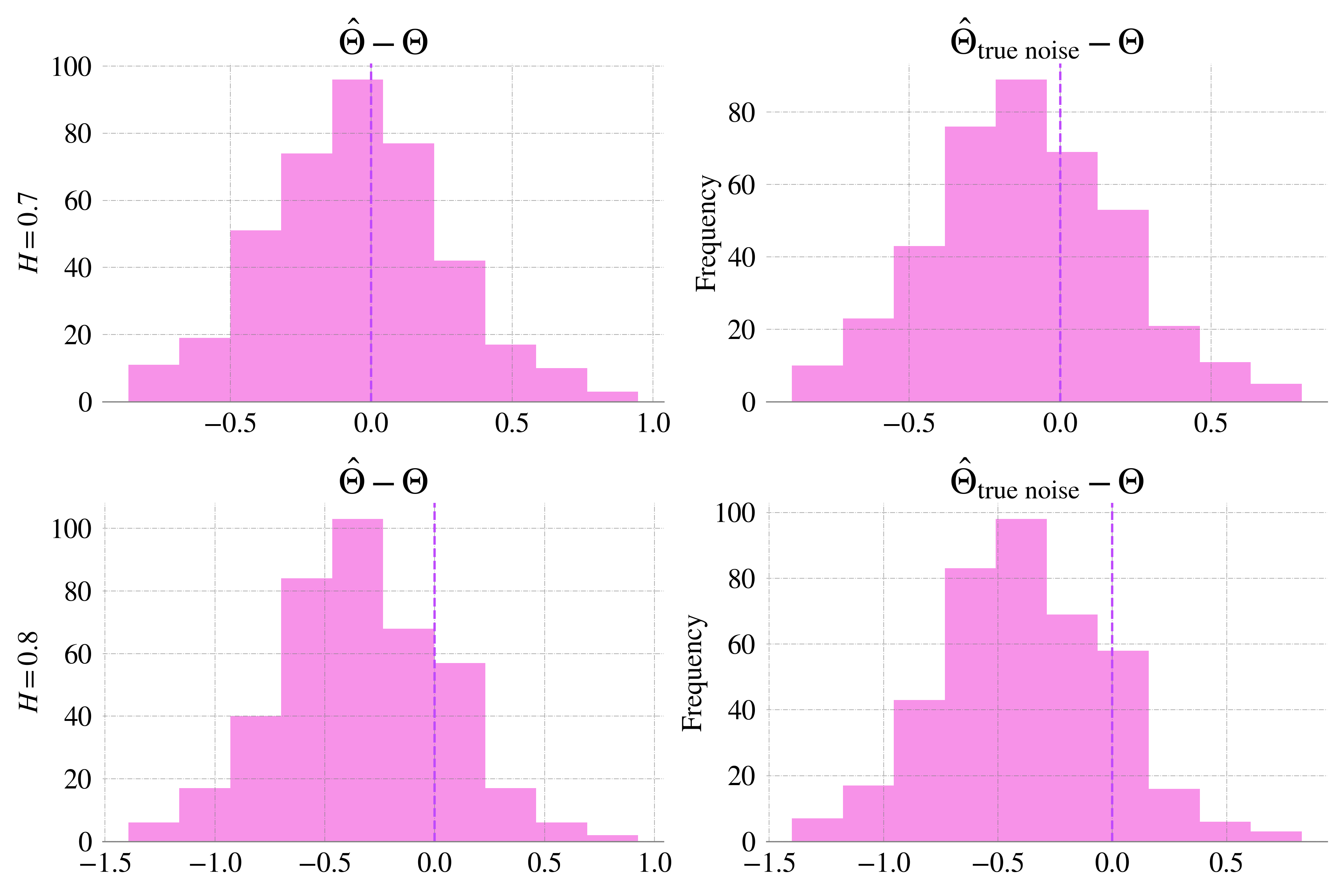}
\caption{Differences between the estimates $\hat \Theta$ from the true parameter and the estimates $\hat \Theta_{\text{true noise}}$ calculated from the real data in the case $\alpha=0.7$.}
\label{fig:theta_hist_0.7}
\end{figure}

\begin{figure}[ht]
     \centering
     \begin{subfigure}{0.9\textwidth}
         \centering
         \includegraphics[width=\textwidth]{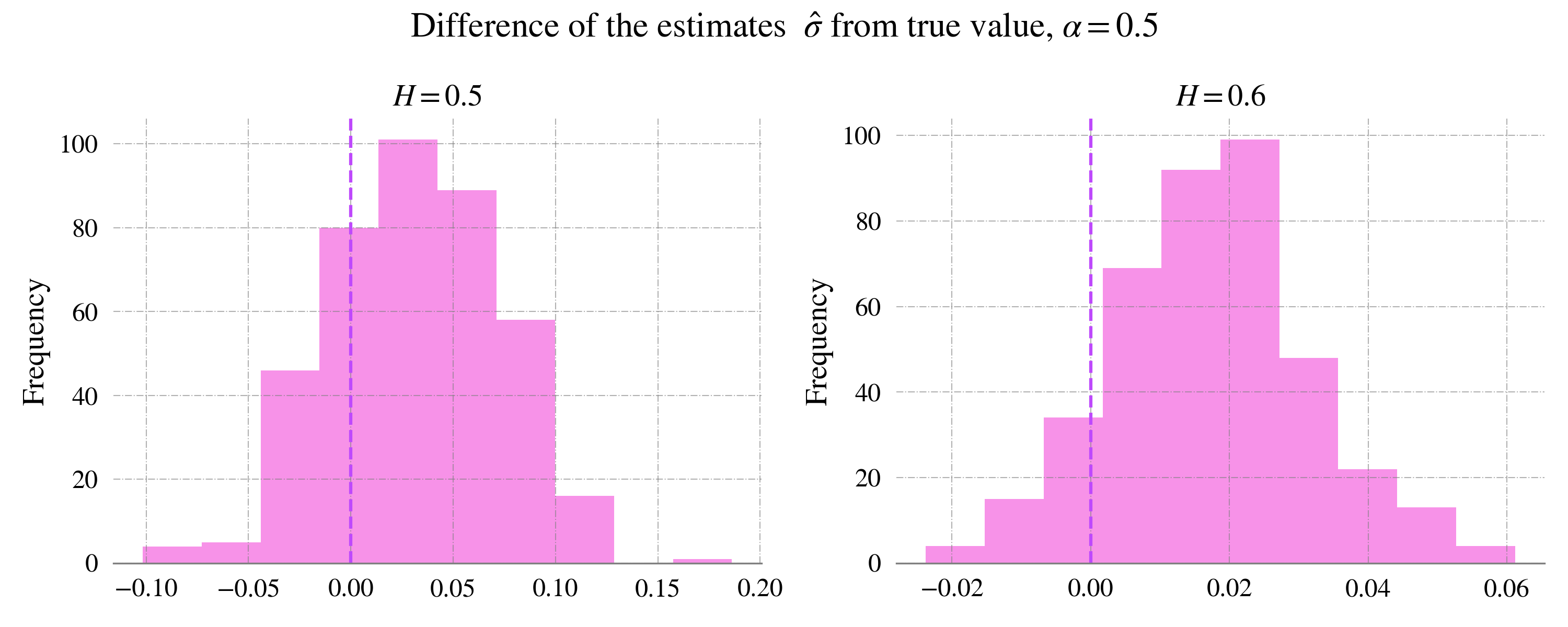}
         \caption{$\alpha = 0.5$.}
         \label{fig:sigma_hist_0.5}
     \end{subfigure}
     \hfill
     \begin{subfigure}{0.9\textwidth}
         \centering
         \includegraphics[width=\textwidth]{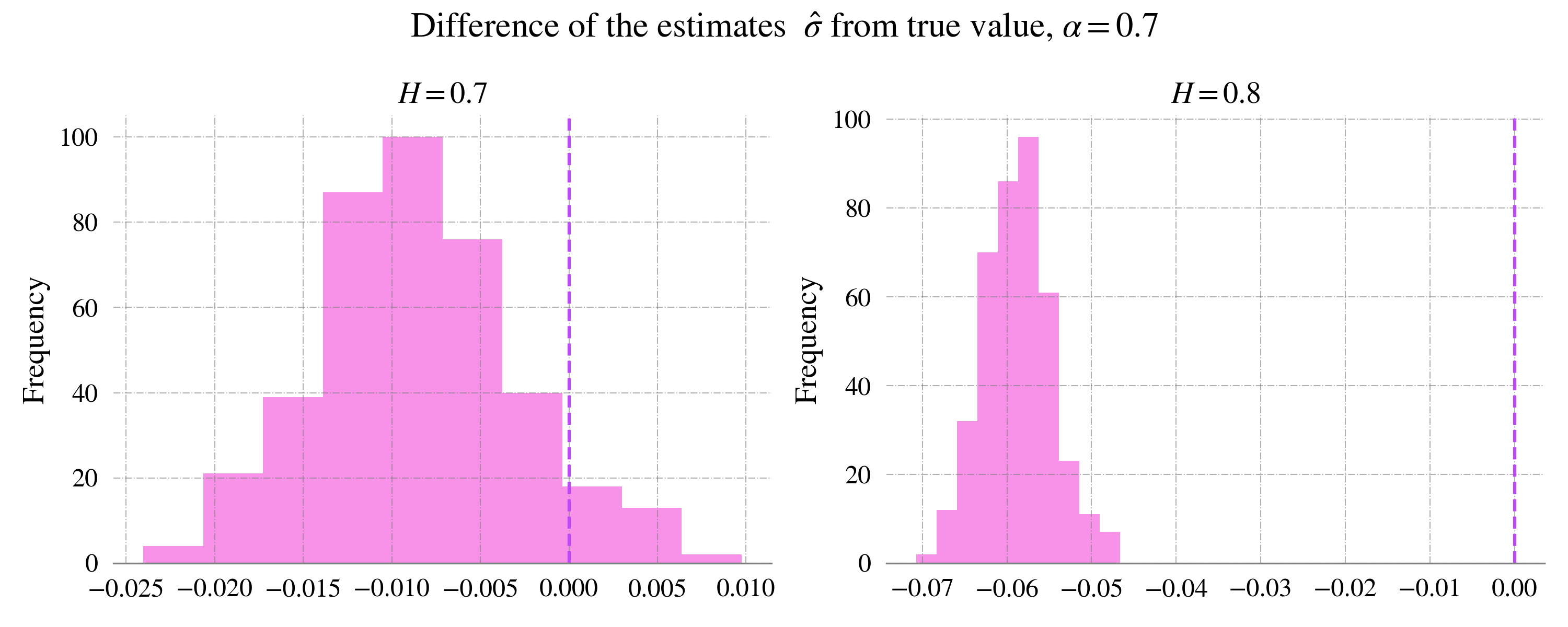}
         \caption{$\alpha=0.7$.}
         \label{fig:sigma_hist_0.7}
     \end{subfigure}
        \caption{Differences between the estimates of $\sigma$ from the true value for different values of $\alpha$ and $H$.}
        \label{fig:sigma_hist}
\end{figure}

\section{Conclusions}
\label{sec:conclusions}

In this article we have investigated a generalized stochastic fractional neuronal model
combining fractional dynamics with correlated stochastic inputs. The proposed framework
extends several classes of fractional neuronal systems previously introduced in the
literature and incorporates, within a unified setting, both hereditary effects induced
by fractional differentiation and temporal dependence generated by stochastic processes
with stationary increments.
More precisely, the model couples a fractional differential equation for the neuronal
state variable with a latent mean-reverting stochastic process describing correlated
external inputs. This formulation allows one to reproduce nonlocal temporal behavior,
persistent correlations, and non-exponential relaxation phenomena which are frequently
observed in biological and neuronal systems and which cannot be adequately captured
within classical Markovian approaches.

The main purpose of the article was not only the introduction of the model itself,
but also the development of a feasible statistical methodology for parameter estimation. For this we have proposed a two-step estimation
procedure. First, the parameters governing the fractional dynamics were estimated
through asymptotic properties of the Mittag--Leffler solution near the origin. Then,
after reconstruction of the latent stochastic input through fractional differentiation,
the parameters of the hidden noise dynamics were estimated by means of inference
techniques for generalized Ornstein--Uhlenbeck type processes.

From the theoretical viewpoint, the analysis highlights the delicate interplay between
the regularity of the driving noise and the order of the fractional derivative. In
particular, the condition
$\alpha < H$
naturally appears in the study of the reconstruction error of the latent process when
the driving noise possesses $H$-H\"older continuous trajectories. Under this regime,
the fractional differentiation operator remains sufficiently stable with respect to
small perturbations of the estimated fractional parameter, allowing a consistent
reconstruction of the hidden stochastic component on compact intervals.

The numerical experiments confirm the practical applicability of the proposed
methodology. The estimation procedure for the parameters $\alpha$ and $A$ performs
well when sufficiently many observations close to the origin are available, consistently
with the asymptotic structure of the solution. Moreover, the recovery of the latent
noise process appears sufficiently accurate to allow subsequent estimation of the
parameters $\Theta$ and $\sigma$ governing the hidden dynamics.
At the same time, the simulations also reveal intrinsic difficulties associated with
non-Markovian models with memory. In particular, the quality of the estimation of
$\Theta$ and $\sigma$ deteriorates as the regularity parameter $H$ increases, reflecting
the stronger dependence structure of the driving noise. More generally, the analysis
suggests that the reconstruction error of the latent process may accumulate over long
time intervals, while reliable estimation of covariance-based quantities requires large
observation horizons. This creates a nontrivial balance between local estimation near
the origin and long-time statistical inference.

From a broader perspective, the present work further supports the relevance
of fractional stochastic analysis in mathematical neuroscience. The combination of
fractional operators, correlated stochastic forcing, and latent-variable inference
provides a flexible and mathematically rich framework for modeling systems with memory
and complex temporal organization. In this regard, the ideas developed may
stimulate further research at the intersection of fractional calculus, stochastic
analysis, statistical inference, and neuronal modeling. As simple examples of open remaining problems, it would be interesting to study exact asymptotic analysis of the estimators and reconstructed noise process. Another natural extension concerns multidimensional neuronal systems and
interacting network models, where both the state dynamics and the latent noise process
are vector-valued. Finally, we restricted our analysis to the case $H>\alpha$. It would be interesting to examine what happens in the regime $H\leq \alpha$.


\begin{thebibliography}{9}

	\bibitem{AbundoPirozzi2021}
	M.~Abundo and E.~Pirozzi,
	\newblock Fractionally integrated Gauss--Markov processes and applications,
	\newblock {\em Communications in Nonlinear Science and Numerical Simulation},
	101 (2021), 105862.
\bibitem{Burk}	
A. N. Burkitt. 
\newblock A review of the integrate-and-fire neuron model: I. homogeneous synaptic input. \newblock {\em Biological Cybernetics}, (2006) 95:1--19.

	


    	\bibitem{Vasicek2025}
	P.~Ilmonen, M.~Laurikkala, K.~Ralchenko, T.~Sottinen and L.~Viitasaari,
	\newblock Data driven modeling of multiple interest rates with generalized
	Vasicek-type models,
	\newblock {\em arXiv preprint arXiv:2509.03208}, 2025.

	

	
	\bibitem{LeonenkoPirozzi2025}
	N.~Leonenko and E.~Pirozzi,
	\newblock The time-changed stochastic approach and fractionally integrated
	processes to model the actin--myosin interaction and dwell times,
	\newblock {\em Mathematical Biosciences and Engineering},
	22 (2025), no.~4, 1019--1054.
	

	
    \bibitem{Magin2006}
	R.~L.~Magin,
	\newblock {\em Fractional Calculus in Bioengineering},
	\newblock Begell House Publishers, Connecticut, 2006.
	
	\bibitem{Mainardi2010}
	F.~Mainardi,
	\newblock {\em Fractional Calculus and Waves in Linear Viscoelasticity},
	\newblock Imperial College Press, London, 2010.
	
	\bibitem{MeerschaertSikorskii2012}
	M.~M.~Meerschaert and A.~Sikorskii,
	\newblock {\em Stochastic Models for Fractional Calculus},
	\newblock De Gruyter, Berlin, 2012.
	
	\bibitem{MeerschaertStraka2013}
	M.~M.~Meerschaert and P.~Straka,
	\newblock Inverse stable subordinators,
	\newblock {\em Mathematical Modelling of Natural Phenomena},
	8 (2013), no.~2, 1--16.
	
	\bibitem{MetzlerKlafter2000}
	R.~Metzler and J.~Klafter,
	\newblock The random walk's guide to anomalous diffusion:
	a fractional dynamics approach,
	\newblock {\em Physics Reports},
	339 (2000), no.~1, 1--77.
	
	\bibitem{Pirozzi2018}
	E.~Pirozzi,
	\newblock Colored noise and a stochastic fractional model for correlated inputs
	and adaptation in neuronal firing,
	\newblock {\em Biological Cybernetics},
	112 (2018), no.~1--2, 25--39.

    
	\bibitem{Pirozzi2024}
	E.~Pirozzi,
	\newblock Mittag--Leffler fractional stochastic integrals and processes with applications,
	\newblock {\em Mathematics},
	12 (2024), no.~19, 3094.

    \bibitem{Pirozzi2024b}
	E.~Pirozzi,
    \newblock
	Some Fractional Stochastic Models for Neuronal Activity with Different Time-Scales and Correlated Inputs.
    \newblock {\em Fractal and Fractional.}  (2024) 8(1):57.  

    \bibitem{SacerdoteGiraudo2013}
L.~Sacerdote and M.~T.~Giraudo,
\newblock Stochastic integrate and fire models: a review on mathematical 
    methods and their applications,
\newblock In {\em Stochastic Biomathematical Models}, Lecture Notes in 
    Mathematics, vol.~2058, Springer, Heidelberg, 2013, pp.~99--148.
    
	\bibitem{Samko1993}
	S.~G.~Samko, A.~A.~Kilbas and O.~I.~Marichev,
	\newblock {\em Fractional Integrals and Derivatives: Theory and Applications},
	\newblock Gordon and Breach Science Publishers, Amsterdam, 1993.
		
	
	\bibitem{Teka2014}
	W.~Teka, T.~M.~Marinov and F.~Santamaria,
	\newblock Neuronal spike timing adaptation described with a fractional
	leaky integrate-and-fire model,
	\newblock {\em PLoS Computational Biology},
	10 (2014), no.~3, e1003526.

    \bibitem{Tuckwell1988a}
H.~C.~Tuckwell,
\newblock {\em Introduction to Theoretical Neurobiology},
\newblock Cambridge University Press, Cambridge, 1988.

	\bibitem{Uchaikin2013}
	V.~V.~Uchaikin,
	\newblock {\em Fractional Derivatives for Physicists and Engineers},
	\newblock Springer, Heidelberg, 2013.
	
\bibitem{Marko1}
M.~Voutilainen, L.~Viitasaari, P.~Ilmonen, S.~Torres, and C.~Tudor,
\newblock Vector-valued generalised {O}rnstein--{U}hlenbeck processes: 
    properties and parameter estimation,
\newblock {\em Scandinavian Journal of Statistics}, 49(3):992--1022, 2022.








            
            \end{thebibliography}
\end{document}